\title{The graph theoretic moment problem}
\author{{\sc L\'aszl\'o Lov\'asz}\footnote{Research supported by
OTKA grant No. 77780 and ERC Advanced research grant No. 227701},
E\"otv\"os Lor\'and University, Budapest\\and\\
{\sc Bal\'azs Szegedy}, University of Toronto, Toronto}
\date{Oct 2010}

\documentclass[11pt]{article}
\usepackage{amssymb,amsmath,graphicx,ntheorem}
\usepackage{hyperref}

\def\FF{{\cal F}}
\def\GG{{\cal G}}\def\II{{\cal I}}
\def\MM{{\cal M}}\def\PP{{\cal P}}\def\QQ{{\cal Q}}
\def\TT{{\cal T}}\def\WW{{\cal W}}

\def\E{{\sf E}}
\def\Var{{\sf Var}}
\def\Pr{{\sf P}}

\def\R{{\mathbb R}}
\def\Z{{\mathbb Z}}
\def\N{{\mathbb N}}

\def\eps{\varepsilon}

\def\hom{{\rm hom}}
\def\inj{{\rm inj}}

\def\rk{{\rm rk}}

\def\ize{{Z^{(2)}}}

\newtheorem{theorem}{Theorem}[section]
\newtheorem{prop}[theorem]{Proposition}
\newtheorem{lemma}[theorem]{Lemma}

\newtheorem{corollary}[theorem]{Corollary}
\theorembodyfont{\rmfamily}
\newtheorem{remark}[theorem]{Remark}

\newtheorem{property}{Property}

\long\def\killtext#1{}

\newenvironment{proof}{\medskip\noindent{\bf Proof.}}{\hfill$\square$\medskip}
\newenvironment{proof*}[1]{\medskip\noindent{\bf Proof of #1.}}{\hfill$\square$\medskip}

\begin{document}

\maketitle

\tableofcontents

\begin{abstract}
We study an analogue of the classical moment problem in the framework
where moments are indexed by graphs instead of natural numbers. We
study limit objects of graph sequences where edges are labeled by
elements of a topological space. Among other things we obtain
strengthening and generalizations of the main results of previous
papers characterizing reflection positive graph parameters, graph
homomorphism numbers, and limits of simple graph sequences. We study
a new class of reflection positive partition functions which
generalize the node-coloring models (homomorphisms into weighted
graphs).
\end{abstract}

\section{Introduction}

To study very large graphs, a natural way to obtain information about
them is sampling. In the case of dense simple graphs, a natural way
to sample is to pick $k$ random nodes and look at the subgraph
induced by them. A sequence $G_1,G_2,\dots$ of simple graphs with
$|V(G_n)|\to\infty$ is called {\it convergent} if the distribution of
this random induced subgraph is convergent for every $k$. To every
convergent sequence of simple graphs one can assign a limit object in
the form of a 2-variable real function \cite{LSz1}.

Instead of the induced subgraph samples, one can consider
homomorphism densities of various ``small'' graphs. While for simple
graphs they trivially carry the same information as the samples
described above (connected by a simple inclusion-exclusion), their
algebraic properties are quite different and often more useful. These
densities are very good 2-variable analogues of moments of 1-variable
functions (see Section \ref{MOMSIMPLE}).

It turns out that in a more general setting, moment sequences can be
indexed by multigraphs rather than simple graphs. Let $X$ be a random
variable. A moment of $X$ (in a slightly generalized sense) is the
expected value of $p(X)$ where $p$ is a polynomial in $\R[x]$. The
classical moment problem can be phrased as follows: which functions
$\alpha:~\R[x]\to\R$ can be represented by a real valued random
variable $X$ so that $\alpha(p)=E(p(X))$ for all $p\in\R[x]$. The
necessary and sufficient condition is that $\alpha$ is linear,
normalized ($\alpha(1)=1$) and positive definite ($\alpha(p^2)\geq 0$
for very polynomial $p$).

Consider a symmetric measurable 2-variable function $W:~[0,1]^2\to
[0,1]$. Let $X_1,X_2,X_3,...$ be random independent elements from
$[0,1]$. The random variables $Z_{i,j}=W(X_i,X_j)~~(i\neq j)$ have
all the same distribution but they are not all independent (for
example, $Z_{1,2}$ and $Z_{2,3}$ are correlated in general). Note
that by the symmetry of $W$, we have $Z_{i,j}=Z_{j,i}$ for every $i$
and $j$.

It is natural to define the moments of $W$ as expected values of
multivariate polynomials in the variables $Z_{i,j}$. As in the
one-variable case, $W$ induces a linear map from the polynomial ring
$\R[\{z_{i,j}|1\leq i<j\}]$ to the real numbers by
\begin{equation}\label{TDEF-P}
t(p,W)=E(p(\{Z_{i,j}|1\leq i<j\})),
\end{equation}
and this moment function is determined by its values on monomials.
Every monomial in this ring corresponds to a multigraph, and if two
such monomials correspond to isomorphic graphs, then the moment
function has the same value on them.

So, just like in the one-variable case, $W$ has a countable number of
``moments'', but instead of forming a single sequence, they are
indexed by (finite) multigraphs.

\subsection{Moments indexed by simple graphs}\label{MOMSIMPLE}

Somewhat surprisingly, if we want to define moments of a $2$-variable
function $W$, it is often enough to restrict ourselves to simple
graphs (in other words, to multilinear polynomials $p\in \ize$). In
this section we recall various results that can be viewed as
supporting this claim. (We'll return to why moments indexed by
multigraphs are needed, and how to treat them.)

Recall that a {\it graph parameter} is a map from the set of finite
graphs to the real numbers, invariant under isomorphism. A {\it
simple graph parameter} is only defined on simple graphs.

Let $\WW$ be the space of bounded symmetric measurable functions
$W:~[0,1]^2\to\R$, and let $F$ be a simple graph with $k$ nodes. We
define
\begin{equation}\label{EQ:T-DEF}
t(F,W)=\int_{[0,1]^{V(F)}} \prod_{ij\in E(F)} W(x_i,x_j)\,dx.
\end{equation}
We call $t(F,W)$ as the {\it $F$-moment} of the function $W$. While
this definition is meaningful for every (multi)-graph $F$, we'll
restrict our attention for the time being to simple graphs.

There is an obvious relation between these moments: if $F_1$ and
$F_2$ are two graphs and $F_1F_2$ denotes their disjoint union, then
\[
t(F_1F_2,W)=t(F_1,W)t(F_2,W).
\]
We call this relation the {\it multiplicativity} of the moments.
Using this relation, we can restrict our attention to moments defined
by connected graphs.

Let us compare some basic properties of these moments with the
analogous properties of moments of one-variable functions.

\begin{property}\label{P1}
{\it Moment sequences are interesting.} For example, the Fibonacci
sequence is a moment sequence. Moment parameters are also
interesting. The number of $q$-colorings of a graph $F$, divided by
$q^{|V(F)|}$, is a moment parameter; more generally, the number of
homomorphisms $\hom(F,G)$ of a graph $F$ into a fixed (for
simplicity, simple) graph $G$ (appropriately normalized) is a moment
parameter. To be precise, if
\[
t(F,G)=\frac{\hom(F,G)}{|V(G)|^{|V(F)|}},
\]
then $t(F,G)=t(F,W_G)$ for an appropriate function $W_G$. The number
of nowhere-zero $k$-flows is an important graph parameter
representable this way.

To show a moment sequence of a non-step-function with combinatorial
significance, let us quote the following example from \cite{LSz1}:
the number
\[
2^{|E(F)|} t(F, \cos(2\pi(x-y)))
\]
is the number of eulerian orientations of the graph $F$.
\end{property}

\begin{property}\label{P2}
{\it Any finite number of moments are independent: no finite
number of moments determine any other.} This is also true in the
2-variable case: {\it For any finite set $F_1,\dots,F_k$ of connected
graphs, the set of vectors $(t(F_1,W),\dots,t(F_k,W))$ has a nonempty
interior in $\R^k$} (Erd\H{o}s, Lov\'asz and Spencer \cite{ELS}).
This shows that each of this countable, but ``large'' set of moments
carries information that is not implied by a finite number of others.
So in a sense this large set of moments is indeed needed (instead of,
say, a two-parameter family).
\end{property}

\begin{property}\label{P3}
{\it The moments determine the function up to a measure preserving
transformation of the variable.} (For one-variable functions, this is
equivalent to saying that they determine the distribution of the
function values, but this would be too weak for two-variable
functions.)  To be more precise, it is well known that if
$f,g:~[0,1]\to\R$ are two (for simplicity, bounded) measurable
functions such that $\int_0^1 f^k=\int_0^1 g^k$ for all $k$, then
there is a third bounded measurable function $h:~[0,1]\to\R$ and
measure-preserving maps $\varphi,\psi:~[0,1]\to[0,1]$ such that
$f(x)=h(\varphi(x))$ and $g(x)=h(\psi(x))$ for almost all $x$.

This fact generalizes to two-variable functions (Borgs, Chayes and
Lov\'asz \cite{BCL}): {\it If $U,W\in \WW$ such that for every simple
graph $F$, $t(F,U)=t(F,W)$, then there exists a function $V\in\WW$
and two measure preserving maps $\varphi,\psi:~[0,1]\to[0,1]$ such
that $U(x,y)=V(\varphi(x),\varphi(y))$ and
$W(x,y)=V(\psi(x),\psi(y))$ almost everywhere.}
\end{property}

\begin{property}\label{P4}
{\it Moment sequences can be characterized by inclusion-exclusion.}
Hausdorff \cite{Haus} proved that a sequence $(a_0,a_1,\dots)$ is the
moment sequence of a function $f$ with $0\le f\le 1$ if and only if
$a_0=1$, and the following inequality holds for all $0\le k\le n$:
\[
\sum_{j=0}^k (-1)^{k-j} {k\choose j} a_{n+j}\ge 0.
\]
(cf. Diaconis and Freedman \cite{DF}).

The following analogue of this for graph parameters was proved by the
authors in \cite{LSz1}: {\it A simple graph parameter $f$ can be
represented as $f=t(.,W)$ with some $W\in\WW_0$ if and only if
$f(K_1)=1$, $f$ is multiplicative, and the following inequality holds
for all simple graphs $F$:}
\[
\sum_{F'\supseteq F\atop V(F')=V(F)} (-1)^{|E(F')\setminus E(F)|}
f(F) \ge 0.
\]
\end{property}

\begin{property}\label{P5}
{\it Moment sequences can be characterized by a semidefiniteness
condition.} Hausdorff gave another characterization as well: a
sequence $(a_0,a_1,\dots)$ is the moment sequence of a function $f$
with $0\le f\le 1$ if and only if $a_0=1$, and the (infinite) matrix
$A$ defined by $A_{ij}=a_{i+j-2}$ $(i,j=1\dots \infty)$ is positive
semidefinite.

An analogue for graph parameters was proved by the authors in
\cite{LSz1}. We need to define what replaces adding up indices $i$
and $j$. To this end, we define {\it $k$-labeled simple graph} ($k\ge
0$) is a finite graph in which $k$ nodes are labeled by $1,2,\dots k$
(it can have any number of unlabeled nodes). The {\it simple product}
$F_1F_2$ of two $k$-labeled graphs $F_1$ and $F_2$ is defined by
taking their disjoint union, and then identifying nodes with the same
label; if we get parallel edges, then their multiplicity is
suppressed. (For 0-labeled graphs product means disjoint union.)

Let $f$ be any simple graph parameter and $k\ge 0$. We define the
following (infinite) matrix $M(f,k)$. The rows and columns are
indexed by isomorphism types of $k$-labeled simple graphs. The entry
in the intersection of the row corresponding to $F_1$ and the column
corresponding to $F_2$ is $f(F_1F_2)$.

With this notation, we can state the following characterization of
moment parameters \cite{LSz1}: {\it A simple graph parameter $f$ can
be represented as $f=t(.,W)$ with some $W\in\WW_0$ if and only if
$f(K_1)=1$, $f$ is multiplicative, and the (infinite) matrix $M(f,k)$
is positive semidefinite for each $k$.}
\end{property}

\begin{property}\label{P6}
{\it A sequence is the moment sequence of a stepfunction if and only
if the matrix $A$ defined above is semidefinite and has finite rank.}
To state an analogous assertion for two-variable functions, we call a
symmetric measurable function $W:~[0,1]^2\to[0,1]$ is a {\it
stepfunction} if there is a finite partition $[0,1]=\cup_{i=1}^r S_i$
into measurable sets such that $W$ is constant on every $S_i\times
S_j$. The following was proved for simple graph parameters by
Lov\'asz and Schrijver \cite{LSch} (paralleling an earlier result by
Freedman, Lov\'asz and Schrijver \cite{FLS} for multigraph
parameters, see Theorem \ref{THM:FLS} below): {\it A simple graph
parameter is the moment parameter of a stepfunction with $q$ steps if
and only if the matrix $M(f,k)$ is semidefinite and has rank at most
$q^k$ for every $k\ge 0$.}

Considering stepfunctions points at other interesting analogies with
the one-variable case. It is not hard to see that {\it a one-variable
function is a stepfunction if and only if it is determined by a
finite set of its moments.} The ``only if'' part of the analogous
statement for 2-variable functions was proved (in graph-theoretic
terms) for two-variable functions by Lov\'asz and S\'os \cite{LS}:
{\it For every stepfunction $U\in\WW$ there is a finite set
$F_1,\dots,F_m$ of simple graphs such that if $t(F_j,U)=t(F_j,W)$ for
some $W\in\WW$ for $j=1,\dots,m$, then $t(F,U)=t(F,W)$ for every
simple graph $F$.} However, the converse fails to hold \cite{LSz5}.
\end{property}

\begin{property}\label{P7}
{\it Convergence in moments implies convergence.} More exactly, if
$X_1,X_2,\dots$ are uniformly bounded random variables such that
$\E(X_n^k)$ is convergent for every $k$, then $X_n$ tends to a limit
in distribution. Analogously, if $(W_n)$ is a uniformly bounded
sequence of functions in $\WW$, then $t(F,W_n)$ is convergent for
every simple graph $F$ if and only if there are measure preserving
maps $\varphi_n:~[0,1]\to[0,1]$ such that the functions
$W_n'(x,y)=W_n(\varphi_n(x),\varphi_n(y))$ are convergent in an
appropriate norm (the $\|.\|_\square$.

This fact is closely related to limits of graph sequences. In fact,
if $(G_n)$ is a sequence of simple graphs for which $t(F,G_n)$ is
convergent for every simple graph $F$, then there is a function
$W\in\WW$ such that $t(F,G_n)\to t(F,W)$ for every $F$ \cite{LSz1}.
This result can be extended to the case when $(G_n)$ is a sequence of
weighted graphs with uniformly bounded edgeweights \cite{BCLSV}.
\end{property}

\subsection{Moments indexed by multigraphs}\label{MOMMULTI}

We have seen that the densities of simple graphs in symmetric
measurable functions $W:~[0,1]^2\to[0,1]$ can be considered as an
analogue of moments. The formula \ref{EQ:T-DEF} defining $F$-moments
makes sense for all (multi)graphs $F$, and there are many reasons why
we don't want to restrict ourselves to just simple graph moments. For
example, we may be interested in the ``ordinary'' moments of a
function $W\in\WW$ (considered as a function in a single variable
defined on the probability space $[0,1]^2$, rather than a 2-variable
function). These moments can be expressed as
\[
\int_{[0,1]^2} W(x,y)^n\,dx\,dy = t(K_2^n,W),
\]
where $K_2^n$ consists of two nodes connected by $n$ parallel edges.
By Property \ref{P3}, this is determined by the simple graph moments,
but it can be seen (using a slight extension of the results mentioned
in Property \ref{P2}) that no finite number of them determines
$t(K_2^n,W)$.

Another reason for considering multigraphs is that we want to think
of a polynomial $p$ in variables $z_{i,j}$ ($1\le i<j\le 1$) as a
formal linear combination of multigraphs. Then every multigraph
parameter $f$ can be extended linearly to these polynomials.

\subsubsection{Limits of weighted graphs}

Suppose that the sequence $t(F,G_n)$ is convergent for every
multigraph $F$ (rather than for every simple graph $F$). Does this
imply that there exists a limit function $W$ that encodes the
limiting values?

To illustrate the difficulty, let $G_n$ be a random graph on $n$
nodes, with edge probability $1/2$. It is easy to see that with
probability $1$,
\[
t(F,G_n)\to 2^{-|E(F)|}=t(F,1/2) \qquad (n\to\infty)
\]
for every simple graph $F$ (here $1/2$ denotes the identically $1/2$
function). It can be shown that this is the only limit function (e.g.
by Property 3 above).

Suppose that $F$ has multiple edges, and let $F'$ denote the simple
graph obtained from $F$ by suppressing the edge multiplicities. Then
$t(F,G_n)=t(F',G_n)$, so $t(F,G_n)\to2^{-|E(F')|}$; but
$t(F,1/2)=2^{-|E(F)|}$, so while the sequence $(t(F,G_n))$ is
convergent for every multigraph $F$, its limit is not $t(F,1/2)$ if
multiple edges are present. By the uniqueness of the limit function,
this means that the limit cannot be described by a single function in
$\WW$.

In \cite{LSz6}, limit objects for moments indexed by multigraphs with
bounded edge multiplicities are described. Let $\WW(d)$ denote the
set of symmetric measurable functions $W:~[0,1]\times [0,1]\to
[-d,d]$. A {\it moment function sequence} is a sequence
$(W_0,W_1,\dots)$ of functions such that $W_i\in\WW(d^i)$ and
$(W_0(x,y),W_1(x,y),\dots)$ is a moment sequence for almost all pairs
$(x,y)\in[0,1]^2$. Then the limit object of a graph sequence with
edge weights uniformly bounded by $d$ can be described by a moment
function sequence. (As in the one-variable and also in the
simple-graph case, these objects are not uniquely determined by their
moments since any measure preserving transformation of $[0,1]$ yields
another object which has the same moments.)

It is also shown in \cite{LSz6} that moment function sequences can be
represented essentially uniquely by functions $W:~[0,1]^2\to\PP(d)$,
where $\PP[-d,d]$ is the set of probability distributions on the
Borel sets of $[-d,d]$ (we endow $\PP[-d,d]$ with the week topology,
and require $W$ to be measurable as a map into the Borel sets of
$\PP[-d,d]$). Such a function is called a {\it $[-d,d]$-graphon}.

There is a third representation which is unique and is analogous to
the distribution of a random variable: This is a probability
distribution on infinite edge-weighted graphs on the node set $\N$
that is symmetric under the permutations of the node set and has the
property that disjoint subsets of $\N$ span independent (in the
probability sense) labeled weighted graphs. (Again, the edge weights
are between $d$ and $-d$.) The above can be viewed as a natural
characterization of these homogeneous infinite random graph models.

\subsubsection{Characterizing moment parameters}

One of the goals of this paper is to characterize moment parameters
indexed by multigraphs. Here are some basic properties of graph
parameters of the form $t(.,W)$, where $W\in\WW(d)$ (see Proposition
\ref{PROP:NEC}):

\begin{itemize}
\item[(1)] $t(K_1,W)=1$ where $K_1$ is the one-node graph ($t(.,W)$
is normalized).

\item[(2)] $t(F_1\cup F_2,W)=t(F_1,W)t(F_2,W))$ for all $F_1$ and
$F_2$, where $F_1\cup F_2$ is the disjoint union of $F_1$ and $F_2$
(multiplicativity).

\item[(3)] $t(p^2,W)\ge 0$ for all polynomials $p\in \ize$ (weak
reflection positivity). Note that this makes sense since, as remarked
above, every multigraph parameter extends to polynomials in $\ize$.

\item[(4)] $|t(K_2^n,W)|\le d^n$ (exponentially bounded growth on
the $n$-fold edge).
\end{itemize}

Let $\TT_3(d)$ denote the of multigraph parameters with these four
properties, and let $f\in\TT_3(d)$. Can $f$ be represented as
$t(.,W)$ with some function $W:\Omega\times\Omega\to [-d,d]$? The
parameter $2^{-|E(F')|}$ discussed above shows that these conditions
are not sufficient; however they are not very far from being
sufficient. We will show (Theorem \ref{closure}) that the set of
graph parameters of the form $t(.,W)$ $(W\in\WW(d))$ is dense with
respect to the pointwise convergence in $\TT_3(d)$. We will also show
that graph parameters in $\TT_3(d)$ can be represented by
$[-d,d]$-graphons.

\begin{remark}\label{PDSEMI}
There is another generalization of the classical moment problem, the
theory of positive definite functions on semigroups \cite{BCR,LM}.
Although our context does not entirely fit into the framework of that
theory, we will make use of a theorem about exponentially bounded
positive definite functions \cite{BM} (see \cite{LSch} for results on
semigroups that are related to both that theory and our framework).
\end{remark}

\subsubsection{Homomorphisms and stepfunctions}

One can define the homomorphism number $\hom(F,H)$ from a multigraph
into a weighted graph, as well as connection matrices $M(f,k)$ for
multigraph parameters $f$, analogously to the simple case. The
analogue of Property 5 above holds (Freedman, Lov\'asz and Schrijver
\cite{FLS}):

\begin{theorem}\label{THM:FLS}
A multigraph parameter is of the form $\hom(.,H)$ for some weighted
graph with $q$ nodes if and only if the multigraph connection matrix
$M(f,k)$ is semidefinite and has rank at most $q^k$ for every $k\ge
0$.
\end{theorem}

In this paper we prove extensions of this theorem. To state our
results, we need the notion of a {\it randomly weighted graph}: a
graph whose nodes are weighted with nonnegative real numbers, and
edges are weighted by random variables with values from a finite set
of real numbers. A weighted graph is a special case when all these
distributions are concentrated on a single value. We say that the
randomly weighted graph is proper, if it is not an ordinary weighted
graph.

Multigraph moments $t(F,H)$ of a randomly weighted graph $H$ can be
defined; they will be multiplicative, normalized, reflection positive
graph parameters.

Our main result (Theorem \ref{genfls2}) describes multigraph
parameters $f$ that are multiplicative, normalized, reflection
positive, and whose second connection matrix $M(f,2)$ has finite rank
(it is enough to require the finiteness of certain very simple
submatrices). The theorem gives two alternatives: such a graph
parameter is either

\smallskip

--- of the form $t(.,H)$ for some weighted graph $H$, in which case
$\rk(M(f,k))^{1/k}\to c\ge 1$ as $k\to\infty$, or

\smallskip

--- of the form $t(.,H)$ for some proper randomly weighted graph
$H$, in which case $\rk(M(f,k))^{1/k^2}\to c>1$ as $k\to\infty$.

\smallskip

In particular, the finiteness of the rank of $M(f,2)$ implies the
finiteness of the ranks of all higher connection matrices $M(f,k)$.

\section{Preliminaries}

\subsection{Graphs and homomorphisms}\label{homom}

We consider four types of graphs. A {\it simple graph} is a finite
undirected graph without loops or multiple edges. In a {\it
multigraph} multiple edges are allowed but loop edges are excluded.
The edge set $E(G)$ of a multigraph $G$ is a multiset of unordered
pairs $ij$ where $i,j$ are distinct elements of the node set. A {\it
weighted graph} $H$ on node set $V=V(H)$ is given by an assignment of
positive nodeweights $(\alpha_i:~i\in V)$ and an assignment of real
edgeweights $\beta_{ij}:~i, j\in V)$. We consider $i,j\in V$ as
adjacent if $\beta_{ij}\not=0$. Note that we allow loop edges in
weighted graphs, but if $\beta_{i,i}=0$ for all $1\leq i\leq n$, then
we say that $H$ is loopless. Every multigraph $F$ can be considered
as a weighted graph with nodeweights $1$ and nonnegative integral
edgeweights (multiplicities) $F_{i,j}$.

We say that $H$ is a {\it randomly weighted graph} if its nodes are
weighted by nonnegative real numbers $\alpha_i$, and its edges are
weighted by independent random variables $B_{i,j}$ with finite
distribution. We can also think of randomly weighted graphs as graphs
whose edges are labeled by moment sequences of random variables with
a finite range, showing that these are discrete versions of
$[-d,d]$-graphons.

Also note that ordinary weighted graphs can be regarded as randomly
weighted graphs in which the edgeweights are single-valued random
variables. An important parameter of randomly weighted graph $H$ will
be $p_{i,j}$, the number of values $B_{ij}$ takes with positive
probability, and $p(H)$, the maximum of the $p_{i,j}$. ordinary
weighted graphs are just those random weighted graphs with $p(H)=1$.

Throughout this paper, if we say just {\it graph}, we mean a
multigraph.

For an arbitrary multigraph $F$ and weighted graph $H$, the
homomorphism number from $F$ to $H$ is defined by
\begin{equation}\label{homform}
\hom (F,H)=\sum_{\varphi:V(F)\to V(H)}~\prod_{i\in
V(F)}\alpha_{\varphi(i)}\prod_{(i,j)\in E(F)
}\beta_{\varphi(i),\varphi(j)}.
\end{equation}
Sometimes it is convenient to normalize the graph parameter
$\hom(G,H)$ and to introduce the {\it homomorphism density}
\[
t(F,H)=\frac{\hom (F,H)}{\bigl(\sum_i\alpha_i\bigr)^{|V(F)|}}.
\]
Note that $t(F,H)=\hom (F,H')$ where $H'$ is obtained from $H$ by
dividing the node weights by $\alpha$. A weighted graph is called
{\it normalized} if the sum of its node weights is $1$.

For an arbitrary graph $F$ with $m$ nodes we define an injective
version of these numbers by the formula
\[
\inj(F,H)=\sum_{\varphi:\,V(F)\hookrightarrow V(H)}~\prod_{(i,j)\in
E(F)}\beta_{\varphi(i),\varphi(j)},
\]
where $\varphi$ ranges over all injective functions from $V(F)$ to
$V(H)$. Again, we can normalize to get
\[
t_\inj(F,H)=\frac{\inj(F,H)}{\sigma_{|V(F)|}(\alpha)},
\]
where $\sigma_k(\alpha)$ denotes the $k$-th elementary symmetric
polynomial of the $\alpha_i$.

For a randomly weighted graph we define the homomorphism number
$\hom(F,H)$ as
\begin{equation}\label{homtorwg}
\hom(F,H)=\sum_{\varphi:\,V(F)\to V(H)}~\prod_{i\in
V(F)}\alpha_{\varphi(i)}\prod_{ij\in E(F)}
\E(B_{\varphi(i),\varphi(j)}^{F_{i,j}}).
\end{equation}
Setting $\beta_{i,j,k}=\E(B_{ij}^k)$, we have
\[
\hom(F,H)=\sum_{\varphi:\,V(F)\to V(H)}~\prod_{i\in
V(F)}\alpha_{\varphi(i)}\prod_{ij\in E(F)}
\beta_{\varphi(i),\varphi(j),F_{i,j}}.
\]
Similarly as before, we introduce the scaled version
\[
t(F,H)=\frac{\hom(F,H)}{\bigl(\sum_i\alpha_i\bigr)^{|V(F)|}}.
\]

\begin{remark}\label{REM:RW}
It is not quite evident where to put the expectation in
\eqref{homtorwg}. Moving it further in like
\[
\sum_{\varphi:\,V(F)\to V(H)}~\prod_{i\in
V(F)}\alpha_{\varphi(i)}\prod_{ij\in E(F)}
\E(B_{\varphi(i),\varphi(j)})^{A_{i,j}}
\]
would of course just reduce the issue to an ordinary weighted graph,
where each random variable $B_{i,j}$ is replaced by its expectation.
Moving it further out like
\[
\E\Bigl(\sum_{\varphi:\,V(F)\to V(H)}~\prod_{i\in
V(F)}\alpha_{\varphi(i)}\prod_{ij\in E(F)}
B_{\varphi(i),\varphi(j)}^{A_{i,j}}\Bigr)
\]
would destroy multiplicativity.
\end{remark}

With every normalized randomly weighted graph $H$ we can associate a
$[-d,d]$-graphon $W_H$, by splitting the unit interval into $|V(H)|$
intervals $S_i$ of length $\alpha_i$, and assigning the random
variable $B_{i,j}$ to each point in $S_i\times S_j$. This graphon
$W_H$ has two finiteness properties: $W_H(x,y)$ has a finite range
for all $x$ and $y$, and there are only a finite number of different
distributions $W_H(x,y)$. In the special case when $H$ is a weighted
graph, we get a function $W\in\WW$. It is easy to see that this
representation has the property that for every multigraph $F$,
$t(F,W_H)=\hom (F,H)=t(F,H)$.

\subsection{Quantum graphs and reflection positivity}

Let $\GG_n~~(n=0,1,2,\dots)$ denote the set of multigraphs in which
$n$ different nodes are labeled by the natural numbers
$\{1,2,\dots,n\}$ (the graphs may have an arbitrary number of
unlabeled nodes). Note that $\GG_0$ is the set of (isomorphism
classes of) graphs without labeled nodes. Let $\FF_n\subset\GG_n$
denote the set of graphs whose node set is $\{1,2,\dots,n\}$. For two
graphs $F_1,F_2\in\GG_n$ we define their product $F_1F_2$ as follows:
we take their disjoint union and then we identify nodes with
identical labels.

The set $\GG_n$ endowed with this multiplication forms a commutative
semigroup with a unit element in which $\FF_n$ is a sub-semigroup. We
denote by $\QQ_n$ the semigroup algebra $\R[\GG_n]$ and by $\PP_n$
the semigroup algebra $\R[\FF_n]$. The elements of these algebras are
formal linear combinations of (partially) labeled graphs, and for
this reason we call them {\em quantum graphs}.

Let us fix a number $n$ and let $z_{i,j}~(1\leq i<j\leq n)$ denote
the graph with $V(z_{i,j})=[n]$ with a single edge connecting $i$ and
$j$. It is clear that $\PP_n$ is generated freely by $\{z_{i,j}|1\leq
i<j\leq n\}$ as a commutative algebra and thus it is isomorphic to
the polynomial ring $\R[\{z_{i,j}|1\leq i<j\leq n\}]$. Note that the
monomials of this polynomial ring are in a one-to-one correspondence
with graphs in $\FF_n$.

A {\em graph parameter} is a map from the set of multigraphs to the
real numbers. Any graph parameter $f$ can be extended linearly to the
vector spaces $\QQ_n$ and $\PP_n$ for all $n\ge 0$. We say that $f$
is {\em reflection positive} (resp.\ {\em weakly reflection
positive}) if $f(p^2)\geq 0$ holds for all natural numbers $n$ and
quantum graphs $p\in\QQ_n$ (resp.\ $p\in\PP_n$).

Any graph parameter $f$, as we have seen, extends linearly to
$\QQ_n$. In addition, $f$ induces a bilinear form
$\langle.,.\rangle_f$ on $\QQ_n$ by $\langle p,q \rangle_f=f(pq)$.
This form has the property that $\langle pq,r \rangle_f=\langle p,qr
\rangle_f$. Note that the reflection positivity (resp.\ weak
reflection positivity) of $f$ is equivalent to the positive
semidefinitness of the bilinear forms $\langle.,.\rangle_f$ on the
algebras $\QQ_n$ (resp.\ $\PP_n$). Let
\[
\II(\PP_n,f)=\{x|x\in \PP_n~,~\langle x,\PP_n\rangle_f=0\}.
\]
It is clear that $\II(\QQ_n,f)$ is an ideal of the algebra $\QQ_n$,
and we can consider the factor $\QQ_n/f=\QQ_n/\II(\QQ_n,f)$. Clearly
$\dim(\QQ_n/f)$ is the rank of the bilinear form
$\langle.,.\rangle_f$ on $\QQ_n$. We can carry out these
constructions with $\PP_n$ instead of $\QQ_n$. In the case when
$f=\hom(.,H)$ for some randomly weighted graph $H$, we also denote
$\QQ_n/f$ by $\QQ_n/H$.

The algebras $\QQ_n/f$ and the numbers $\dim(\QQ_n/f)$ were
introduced in \cite{FLS}.

Basic properties of these algebras can also be expressed in terms of
certain matrices. The $n$-th {\em connection matrix} of a graph
parameter $f$ is an infinite matrix $M(n,f)$ whose rows and columns
are indexed by the elements of $\GG_n$ and the entry in the
intersection of the row corresponding to $F_1$ and column
corresponding to $F_2$ is $f(F_1F_2)$. The rank of this matrix is
equal to $\dim(\QQ_n/f)$, and this matrix is positive semidefinite if
and only if so is the bilinear form $\langle.,.\rangle_f$ on $\QQ_n$.

A graph parameter $f$ is {\it multiplicative} if $f(F_1\cup
F_2)=f(F_1)f(F_2)$ where $F_1\cup F_2$ is the disjoint union of $F_1$
and $F_2$. We call $f$ {\it normalized} if takes the value $1$ on a
single node. It is clear that $f$ is multiplicative and normalized if
and only if the induced map $f:~\QQ_0\to\R$ is an algebra
homomorphism. It is also easy to see that $f$ is multiplicative if
and only if $\dim(\QQ_0,f)\leq 1$.

\subsection{Semidefinite functions on polynomial rings}

Let $n$ be a fixed natural number and let $x_1,x_2,\dots x_n$ be
variables. A {\em polynomial expression} of $W$ in $n$ variables is a
polynomial of the functions $\{W(x_i,x_j)~|~1\leq i<j\leq n\}$. Note
that these $n$-variable functions form a commutative algebra with the
pointwise multiplication and addition. We define the {\em moment} of
$W$ corresponding to a polynomial expression $p(x_1,x_2,\dots,x_n)$
as
\[
\int_{[0,1]^n}p~dx_1\,\dots\,dx_n.
\]

The ring of the polynomial expressions of $W$ in $n$ variables is a
homomorphic image of $\PP_n$ where the homomorphism is given by
$z_{i,j}\to W(x_i,x_j)$. Composing the moment map with this
homomorphism we obtain a linear map $t_W:~ \PP_n\to \R$. Since the
moments of a polynomial expression are invariant under any
permutation of the variables, we obtain that for an element
$F\in\FF_n$ the value $t_W(F)=t(F,W)$ does not depend on the labeling
of $F$, only on its isomorphism class. For this reason we can also
regard $t_W$ as a graph parameter which carries all the information
about moments. Furthermore, it is easy to see that $t_W$ is
normalized and multiplicative. Let $n$ be an arbitrary natural number
and $p\in \PP_n$. Since the polynomial expression corresponding to
$p^2$ is the square of the polynomial expression corresponding to $p$
we get that $t_W(p^2)\geq 0,$ which shows that the graph parameter
$t_W$ is weakly reflection positive. If $W\in\WW(d)$ then
\[
t(K_2^k,W ) = \Bigl|\int_{(x,y)\in [0,1]^2}W(x,y)^k\,dx\,dy\Bigr|\leq
d^k,
\]
which is equivalent to $|t_W(z_{1,2}^k)|\leq d^k$.

Let $\alpha$ be an element of the dual space of
$\R[x_1,x_2,\dots,x_n]$. The map $\alpha$ is said to be {\it positive
semidefinite} if $\alpha(f^2)\geq 0$ for all $f\in
\R[x_1,x_2,\dots,x_n]$. We say that $\alpha$ is normalized if
$\alpha(1)=1$. The following theorem follows quickly from the theory
of semidefinite functions on Abelian semigroups.

\begin{theorem}\label{semfuncpol} Let $\alpha$ be a normalized, positive
semidefinite element of the dual space of $\R[x_1,x_2,\dots,x_n]$
such that $|\alpha(x_i^r)|\leq d^r$ for all natural numbers $1\leq
i\leq n$ and $0\leq r$ . Then there is a unique probability measure
$\mu$ on $[-d,d]^n$ with
\begin{equation}\label{mertek}
\alpha(f)=\int_{x\in [-d,d]^n}f(x)~d\mu.
\end{equation}
If the rank of the bilinear form given by $\langle f,g
\rangle=\alpha(fg)$ is finite then the measure $\mu$ is concentrated
on finitely many points.
\end{theorem}

\begin{proof}
Let us introduce the linear function $\beta:
\R[x_1,x_2,\dots,x_n]\to\R$ by
\[
\beta(x_1^{r_1}x_2^{r_2}\dots x_n^{r_n})=d^{-(r_1+r_2+\dots
+r_n)}\alpha(x_1^{r_1}x_2^{r_2}\dots x_n^{r_n}).
\]
It is easy to see that $\beta$ is a positive semidefinite function
and $|\beta(x_i^r)|\leq 1$ for all $i$ and $r$. We show that
$|\beta(x_1^{r_1}x_2^{r_2}\dots x_n^{r_n})|\leq 1$. We do it by
induction on the index of the last nonzero $r_i$. Assume that
$|\beta(x_1^{r_1}x_2^{r_2}\dots x_i^{r_i})|\leq 1$ for all possible
sequences $r_1,r_2,\dots,r_i$. Let $p=x_1^{r_1}x_2^{r_2}\dots
x_i^{r_i}$. It follows by the positive semidefinitness of $\beta$
that
\[
\beta((p\pm x_{i+1}^{r_{i+1}})^2)\geq 0
\]
and so
\[
2\geq\beta(p^2)+\beta(x_{i+1}^{2r_{i+1}})\geq \pm 2px_{i+1}^{r_{i+1}}
\]
which implies $1\geq |px_{i+1}^{r_{i+1}}|$. As a consequence we get
that \[|\alpha(x_1^{r_1}x_2^{r_2}\dots x_n^{r_n})|\leq
d^{r_1+r_2+\dots +r_n}.\] This means that $\alpha$ is an
exponentially bounded positive semidefinite function on the semigroup
of the monomials which is isomorphic to $\N_0^n$. Now 2.5.Theorem
(...) completes the proof.
\end{proof}

Note that any probability measure $\mu$ on $[-d,d]^n$ defines a
semidefinite function $\alpha$ on $\R[x_1,x_2,\dots,x_n]$ by
(\ref{mertek}).

We will need two further well-known facts.

\begin{lemma}\label{rankofbf}
Assume that the map $\alpha:\R[x_1,x_2,\dots,x_n] \mapsto\R$ is
defined by
\[
\alpha(p)=\sum_{i=1}^k h_ip(a_i),
\]
where $a_1,\dots,a_k\in\R^n$ are different real vectors, and the
weights $h_i$ are positive real numbers.  Then $\alpha$ is a positive
semidefinite function and the corresponding bilinear form
\[
\langle p_1,p_2 \rangle
=\alpha(p_1p_2)~~(p_1,p_2\in\R[x_1,x_2,\dots,x_n])
\]
has rank $k$.
\end{lemma}

\begin{lemma}\label{konvmet}
Let $d>0$ be a fixed real number. A sequence of measures
$\mu_1,\mu_2,\dots$ on $[-d,d]^n$ is weakly convergent if and only if
$\lim_{k\to\infty}\int f d\mu_k$ exists for every monomial $f$.
\end{lemma}

\subsection{Two-variable functions as operators}

Any bounded symmetric function $W$ on $[0,1]^2$ gives rise to a
symmetric integral kernel operator $T_W$ on the Hilbert space
$L_2([0,1])$, by
\[
T_W(f)(x)=\int_0^1 W(y,x)f(y)~dy.
\]
It follows from the Hilbert-Smith condition that such an operator is
always compact, and so it has a countable set of nonzero eigenvalues
$\{\lambda_1,\lambda_2,\lambda_3\dots\}$, where we may assume that
$|\lambda_1|\ge|\lambda_2|\ge\dots$. It is known that $\lambda_k\to
0$, and so every nonzero eigenvalue has finite multiplicity. We will
need the well-known fact that for $n\ge 2$,
\begin{equation}\label{EQ:CNLAMBDA}
t(C_n,W)=\sum_{k=1}^{\infty}\lambda_k^n.
\end{equation}

The operator rank of $T_W$ and the matrix rank of $C(T_W)$ are either
both infinite or both finite. More exactly,

\begin{lemma}\label{LEM:TW-RANK}
The rank of $C(t_W)$ is between the number of different nonzero
eigenvalues and the number of all nonzero eigenvalues of $T_W$.
\end{lemma}

\begin{proof}
By \eqref{EQ:CNLAMBDA},
\[
C(t_W)_{i,j}= \sum_k \lambda_k^{i+j}=\sum_k m_k\bar\lambda_k^{i+j}.
\]
If this sum is finite, i.e., $\rk(T_W)=m$ is finite, then $C(t_W)$ is
the sum of $m$ matrices of rank $1$, and so it has rank at most $m$.

Conversely, suppose that $C(t_W)$ has finite rank $n$. Then there is
a linear dependence between its first $n+1$ columns, which means that
we have a relation
\[
\sum_{j=1}^{n+1} a_j \sum_k\lambda_k^{i+j} =0
\]
valid for all $i\ge 1$. We can rewrite this as
\[
\sum_k p(\lambda_k) \lambda_k^{i+1} = 0,
\]
where $p$ is a polynomial of degree at most $n$.

We claim that $p(\lambda_k)=0$ for all $k$. Suppose not, and let $r$
be the first index for which $p(\lambda_r)\not=0$, and let $a$ and
$b$ be the multiplicities of the eigenvalues $\lambda_r$ and
$-\lambda_r$ ($a\ge 1, b\ge 0$). Then we have
\[
a p(\lambda_r) + (-1)^{i+1} b p(-\lambda_r) =
-\sum_{k:|\lambda_k|<|\lambda_r|} p(\lambda_k)
\Bigl(\frac{\lambda_k}{\lambda_r}\Bigr)^{i+1}.
\]
Here the right hand side tens to $0$ as $i\to\infty$, implying that
$a+b=0$ and also $a-b=0$, which is a contradiction.

So every nonzero eigenvalue of $T_W$ is a root of $p$, which means
that their number is at most $\deg(p)\le n$.
\end{proof}

This implies

\begin{corollary}\label{finiterank}
Let $W\in\WW(d)$ and assume that $C(t_W)$ has finite rank. Then the
kernel operator $T_W$ is of finite rank and there is a finite
sequence of pairwise orthogonal functions $g_1,g_2,\dots,g_k\in
L_2([0,1])$ and numbers $\nu_i\in\{d,-d\}$ such that
\[
W(x,y)=\sum_{i=1}^k \nu_ig_i(x)g_i(y)
\]
almost everywhere on $[0,1]^2$.
\end{corollary}

The product of two operators $T_{W_1}$ and $T_{W_2}$ is $T_{W_1\circ
W_2}$ where $W_1\circ W_2$ is given by
\[
(W_1\circ W_2)(x,y)=\int_0^1 w_1(x,z)w_2(z,y)\,dz\,.
\]
Let $F'$ denote the graph which is obtained from $F$ by subdividing
each edge in $E(G)$. It will be useful to note that

\begin{lemma}\label{subdiv}
If $W\in\WW$ and $F$ is any graph then $t(F',W)=t(F,W\circ W)$.
\end{lemma}

\section{Results and proofs}

\subsection{Moments and moment-like graph parameters}

Let $\WW(d)$ denote the set of 2-variable measurable functions
$W:~[0,1]^2\to [-d,d]$ that are symmetric in the sense that
$W(x,y)=W(y,x)$ for all $x,y\in [0,1]$. We denote by $\WW$ the union
of the sets $\WW(d)$ over all real numbers $d$.

Let us define four sets of graph parameters:
\begin{align*}
\TT_0(d)&=\{t(.,H):~H\text{ is a $[-d,d]$-weighted graph}\},\\
\TT_1(d)&=\{t(.,H):~H\text{ is a randomly $[-d,d]$-weighted graph}\},\\
\TT_2(d)&=\{t(.,W):~W\in\WW(d)\},\\
\TT_3(d)&=\{t(.,W):~W\text{ is a $[-d,d]$-graphon}\}.
\end{align*}
Clearly $\TT_0(d)\subseteq \TT_1(d),\TT_2(d)\subseteq \TT_3(d)$. We
prove that equality almost holds here. Let us quote Theorem 2.6 from
\cite{LSz6}, applied to our case:

\begin{theorem}\label{convgen-m}
Let $W_1,W_2,\dots$ be a sequence of $[-d,d]$-graphons such that
$(t(F,W_1),t(F,W_2),\dots)$ is a convergent sequence for every
multigraph $F$. Then there is a $[-d,d]$-graphon $W$ such that
$t(F,W_n)\to t(F,W)$ for every $F$.
\end{theorem}

This theorem implies that $\TT_3(d)$ is a closed subset of the space
of graph parameters under pointwise convergence. We are going to
prove:

\begin{theorem}\label{closure}
The set $\TT_3(d)$ is the closure of $\TT_0(d)$.
\end{theorem}

We are also going to prove

\begin{theorem}\label{chr-t2}
A graph parameter $f$ belongs to $\TT_3(d)$ if and only if it is
normalized, multiplicative, weakly reflection positive and satisfies
$|f(K_2^k)|\leq d^k$.
\end{theorem}

By Tychonov's Compactness Theorem, $\TT_3(d)$ is compact as a closed
subspace of the compact space $\prod_F [-d^{|E(F)|},d^{|E(F)|}]$.

Both theorems will follow if we prove two facts:

\begin{prop}\label{PROP:NEC}
Every graph parameter $f\in\TT_3(d)$ is normalized, multiplicative,
reflection positive and satisfies $|f(K_2^k)|\leq d^k$.
\end{prop}

\begin{prop}\label{PROP:SUFF}
Every normalized, multiplicative, weakly reflection positive graph
parameter satisfying $|f(K_2^k)|\leq d^k$ is the limit of graph
parameters in $\TT_0$.
\end{prop}

Before proving these propositions, we state an easy lemma about
homomorphism densities and injective homomorphism densities, which
follows from Lemma 2.1 in \cite{LSz1} by scaling the edgeweights.

\begin{lemma}\label{tavolsag}
Let $H$ be a weighted graph with $n$ nodes such that all the
nodeweights are $1$ and the edgeweights are in $[-d,d]$. Then for an
arbitrary multigraph $F$ with $m$ nodes,
\[
|t(F,H)-t_\inj(F,H)|\leq 2{{m}\choose{2}}\frac{1}{n}d^{|E(F)|}.
\]
\end{lemma}

\begin{proof*}{Proposition \ref{PROP:NEC}}
It is clear that for every $[-d,d]$-graphon $W$, the parameter
$f=t(.,W)$ is multiplicative and normalized, and satisfies
$|f(K_2^k)|\leq d^k$. We prove that $M(f,n)$ is positive
semidefinite.

Let $F$ be a graph in $\GG_n$ such that $V(F)=[m]$ for some natural
number $m\geq n$. For every choice of the variables $x_1,\dots,x_n$
we define
\[
t_{x_1,\dots,x_n}(F,W)=\int\limits_{[0,1]^{m-n}} \prod_{{1\leq i\leq
m,n+1\leq j\leq m}\atop{i < j}}W_{F_{i,j}}(x_i,x_j)\,dx_{n+1}\dots
dx_m.
\]
We have
\begin{align*}
t(F F',W)=&\int_{[0,1]^n}t_{x_1,\dots,x_n}(F,W)
t_{x_1,\dots,x_n}(F',W)\\
&\times\prod_{1\leq i<j \leq
n}W_{F_{i,j}+F'_{i,j}}(x_i,x_j)\,dx_1\dots dx_n.
\end{align*}
where $F$ and $F'$ are graphs in $\GG_n$.

For every $x\in [0,1]^n$, let $M(x)$ denote the $\GG_n\times\GG_n$
matrix in which
\[
M(x)_{F,F'}=t_x(F,W)t_x(F',W)\prod_{1\leq i<j\leq
n}W_{F_{i,j}+F'_{i,j}}(x_i,x_j).
\]
From the above formulas one obtains that
\begin{equation}\label{EQ:MNF}
M(f,n)=\int_{[0,1]^n} M(x)\,dx,
\end{equation}
so it suffices to prove that $M(x)$ is positive semidefinite for
every $x$.

For $1\le i<j\le n$, let $M(x,i,j)$ denote the $\GG_n\times\GG_n$
matrix in which
\[
M(x,i,j)_{F,F'}=W_{F_{i,j}+F'_{i,j}}(x_i,x_j).
\]
Since $M(x,i,j)$ is essentially (up to repetition of rows and
columns) the moment matrix of the random variable $W(x_i,x_j)$, it is
positive semidefinite. We get $M(x)$ from the Schur product of the
matrices $M(x,i,j)$ over all possible pairs $1\leq i<j\leq n$ by
scaling the rows and columns. This shows that $M(x)$ is indeed
positive semidefinite.
\end{proof*}

\begin{proof*}{Proposition \ref{PROP:SUFF}}
It suffices to consider the case $d=1$, since we can scale the
edgeweights by $1/d$. Let $f$ be a weakly reflection positive,
normalized multiplicative graph parameter with $f(K_2^k)\le 1$ for
all $k\ge 1$. We prove that there is a sequence of stepfunctions
$U_1,U_2,\dots$ in $\WW(1)$ such that $\lim_{n\to
\infty}t(F,U_n)=f(F)$ for every graph $F$.

The weak reflection positivity of $f$ means that $f$ is a
semidefinite function on the polynomial ring $\PP_n$ for every
natural number $n$. Using that $f(K_2^k)\leq 1$ and Theorem
\ref{semfuncpol} for $\PP_n$ we obtain that there is a unique
probability measure $\mu_n$ on $[-1,1]^{{n}\choose{2}}$ such that
\[
f(F)=\E\Bigl(\prod_{1\leq i<j\leq n}z_{i,j}^{F_{i,j}}\Bigr)
\]
for every graph $F\in \FF_n$, where the $z_{i,j}$ are regarded as
random variables whose joint distribution is given by $\mu_n$. Let
$Z_n$ be a random weighted graph (not a randomly weighted graph!) on
$[n]$ with nodeweights $1/n$ and edgeweights $z_{i,j}$. Since $f$ is
invariant under relabeling the nodes of $F$ we get that
\[
f(F)=\E\Bigl(\frac{1}{n!}\sum_{\sigma\in S_n}\prod_{1\leq i<j\leq
n}z_{\sigma(i),\sigma(j)}^{F_{i,j}}\Bigr)=\E(t_\inj(F,Z_n)).
\]

Fix a graph $F\in \FF_m$ and for every $n\geq m$, define the graph
$F_n\in\FF_n$ by adding $n-m$ isolated labeled nodes to $F$. It is
clear that
\[
t_{\inj}(F_n,Z_n)=t_\inj(F,Z_n)
\]
and (using the properties of $f$) that
\begin{equation}\label{injeq}
f(F)=f(F_n)=\E(t_\inj(F_n,Z_n))=\E(t_\inj(F,Z_n))
\end{equation} for all
$n\geq m$. Let $F^2$ denote the disjoint union of $F$ with itself.
Using that $f(F^2)=f(F)^2$ and (\ref{injeq}) we get that
\begin{align*}
\Var(t_\inj(F,Z_n))&=\E(t_\inj(F,Z_n)^2)-\E(t_\inj(F,Z_n))^2\\
&=\E(t_\inj(F,Z_n)^2)-f(F^2)=\E(t_\inj(F,Z_n)^2-t_\inj(F^2,Z_n)) .
\end{align*}
From Lemma \ref{tavolsag} it follows that
\[
|t(F,Z_n)^2-t_\inj(F,Z_n)^2|\leq
\Bigl|2\binom{m}{2}\frac{1}{n}(t(F,Z_n)+t_\inj(F,Z_n))\Bigr|\leq
\frac{4}{n}\binom{m}{2}.
\]
and similarly
\[
|t(F^2,Z_n)-t_\inj(F^2,Z_n)|\leq \frac{2}{n}\binom{2m}{2}.
\]
Using that $t(F^2,Z_n)=t(F,Z_n)^2$, we get
\begin{align*}
|t_\inj(F,Z_n)^2-t_\inj(F^2,Z_n)| \le\frac{6m^2}{n}.
\end{align*}
Thus
\[
\Var(t_\inj(F,Z_n))=\E(t_\inj(F,Z_n)^2-t_\inj(F^2,Z_n))\le
\frac{6m^2}{n}.
\]
By Chebyshev's inequality, we have for every $\eps>0$,
\[
\Pr(|t_\inj(F,Z_n)-f(F)|>\eps)\le \frac{6m^2}{\eps^2n}.
\]
It follows by the Borel-Cantelli lemma that
\[
\lim_{n\to\infty}t_\inj(F,Z_{n^2})=f(F)
\]
with probability $1$.

Since there are only countably many different graphs $F$, we obtain
that the above convergence holds simultaneously for all graphs with
probability $1$. By Lemma \ref{tavolsag} we get that the graph
parameter $t(.,Z_{n^2})$ converges to $f(F)$ with probability one in
the space of graph parameters. Thus $f$ is in the closure of
$\TT_0(d)$.
\end{proof*}

\subsection{Finiteness conditions}

In a sense, the classes $\TT_0(d)$ and $\TT_1(d)$ are finite versions
of the classes $\TT_2(d)$ and $\TT_3(d)$. Theorem \ref{THM:FLS} tells
us that a graph parameter $f\in\TT_2(d)$ belongs to $\TT_0(d)$ if and
only if there is a positive integer $q$ such that $\rk(M(f,k))\le
q^k$ for all $k$. We prove that a much weaker condition is
sufficient.

For a graph parameter $f$, we define three infinite matrices $E(f)$,
$C(f)$ and $B(f)$, in each of which the rows and columns are indexed
by the natural numbers $0,1,2,3,\dots$, and the entries are defined
by three one-parameter families of graphs. Let $K_2^n$ consist of 2
nodes joined by $n$ edges, let $C_n$ be the $n$-cycle, and let
$K_{a,b}$ be the complete bipartite graph with color classes if sizes
$a$ and $b$. We define
\begin{align*}
E(f)_{ij}&=f(K_2^{i+j}),\\
C(f)_{ij}&=f(C_{i+j-1}),\\
B(f)_{ij}&=f(K_{i+j,2}).
\end{align*}
Note that all three matrices are submatrices of the multigraph
connection matrix $M(f,2)$.

\begin{theorem}\label{step}
For a graph parameter $f\in\TT_2(d)$, the following are equivalent:

\smallskip

{\rm (a)} $f\in\TT_0(d)$;

\smallskip

{\rm (b)} both $C(f)$ and $E(f)$ have finite rank;

\smallskip

{\rm (c)} both $B(f)$ and $E(f)$ have finite rank;

\smallskip

{\rm (d)} $M(f,2)$ has finite rank.
\end{theorem}

Even though the graphs used in condition (b) in Theorem \ref{step}
are smaller, condition (c) may be more useful because it doesn't use
graphs with multiple edges.

\begin{proof}
It is trivial that (a) implies (d), which in turn implies both (b)
and (c).

\medskip

\noindent (b)$\Rightarrow$(a). Let $f=t(.,W)$ and let $X_W$ denote
the random variable $W(X_1,X_2)$ where $X_1$ and $X_2$ are chosen
uniformly at random from $[0,1]$. The $n$-th moment of $X_W$ is
\[
\int_{[0,1]^2} W(x_1,x_2)^n\,dx_1\,dx_2=t(K_2^n,W).
\]
Let a linear map $\alpha:~\R[x]\to\R$ be defined by
\[
\alpha(x^n)=t(K_2^n,W).
\]
Since the matrix $E(t_W)$ is the matrix of the bilinear form $\langle
f,g \rangle =\alpha(fg)$ in the basis $1,x,x^2,\dots$ it follows from
Theorem \ref{semfuncpol} that the distribution of $X_W$ is
concentrated on some finite set $S$. Since $C(T_W)$ has finite rank,
Corollary \ref{finiterank} implies that there is a finite system of
one-variable functions $g_1,g_2,\dots,g_k$ and signs
$\nu_i\in\{1,-1\}$ such that
\[
W(x,y)=\sum_{i=1}^k \nu_ig_i(x)g_i(y)
\]
almost everywhere. By changing $W$ on a zero measure set, we can
assume that the previous equality holds everywhere. Since the set
$\{(x,y):~W(x,y)\notin S\}$ has measure $0$, there is a set
$Z\subseteq [0,1]$ with measure $0$ such that for all $x\in
[0,1]\setminus Z$, the function $W(x,.)$ is measurable and the set
$\{y\in[0,1]:~W(x,y)\notin S\}$ has measure $0$. For every fixed $x$,
the function $W(x,.)$ is an element of the finite dimensional
subspace generated by $g_1,g_2,\dots, g_k$, and so there are points
$x_1,x_2,\dots,x_k\in [0,1]\setminus Z$ such that every function
$W(x,.)$, $x\in [0,1]\setminus Z$ is a linear combination of the the
functions $W(x_i,.)$.

For each $i$, there is partition $\{U^i_0,U^i_1,\dots,U^i_s\}$ of
$[0,1]\setminus Z$ into measurable sets, where $s=|S|$, such that
$\lambda(U^i_0)=0$ and $W(x_i,.)$ is constant on each $U^i_j$, $1\le
j\le s$. Combining the sets $U^i_0$ into a single $0$-measure set,
and taking a common refinement of the partitions on the rest, we get
a finite partition $\{U_0,U_1,\dots,U_N\}$ such that $\lambda(U_0)=0$
and each $W(x_i,.)$ is constant on each $U_j$, $1\le j\le N$. Hence
every function $W(x,.)$, $x\in [0,1]\setminus Z$, is constant on
every set $U_j$, $1\le j\le N$. From the symmetry of $W$ it follows
that $W$ is constant on every set $U_i\times U_j$, and so $W$ is
equal to a stepfunction almost everywhere.

\medskip

\noindent (c)$\Rightarrow$(a). It follows from Lemma \ref{subdiv}
that the matrix $B(t_W)$ is the same as $E(t_{W\circ W})$ and that
$C(t_{W\circ W})$ is a submatrix of $C(t_W)$. So $W\circ W$ satisfies
(b), and so we already know that $W\circ W$ is a stepfunction. Thus
$T_{W\circ W}$ has finite rank and every eigenvector of $T_{W\circ
W}$ corresponding to a nonzero eigenvalue is a one-variable
stepfunction. Since $T_{W\circ W}$ is the square of $T_W$, it follows
that the same statement holds for $T_W$. This implies that $W$ is a
stepfunction.
\end{proof}

\subsection{Homomorphisms into randomly weighted graphs}

We prove the following generalization of Theorem \ref{THM:FLS}.

\begin{theorem}\label{genfls}
Let $f$ be a graph parameter. Then the following are equivalent.

\smallskip

{\rm(1)} There is a randomly weighted graph $H$ such that $f(F)=\hom
(F,H)$ for all graphs $F$.

\smallskip

{\rm(2)} $f$ is reflection positive, multiplicative and
$\rk(M(2,f))<\infty$.

\smallskip

{\rm(3)} $f$ is weakly reflection positive, multiplicative and
$\rk(M(2,f))<\infty$.
\end{theorem}

As a corollary, we obtain the following characterization of simple
graph parameters representable as homomorphism functions:

\begin{corollary}
If $f$ is weakly reflection positive, multiplicative
and $\rk(M(2,f))<\infty$, then there is a weighted graph $H$ such
that $\hom (F,H)=f(F)$ for all simple graphs $F$.
\end{corollary}

\begin{proof*}{Theorem \ref{genfls}}
(1)$\Rightarrow$(2) Let $H$ be a randomly weighted graph. We may
scale the nodeweights so that they sum to $1$. Then $f=t(.,W_H)$. By
Proposition \ref{PROP:NEC} we know that $f$ is multiplicative,
normalized and reflection positive. We need to prove that $M(f,2)$
has finite rank.

This follows easily by looking at the proof of Proposition
\ref{PROP:NEC} carefully. In \eqref{EQ:MNF}, the integral can be
replaced by a finite sum with $|V(H)|^n$ terms, since the integrand
depends only on the nodes of $H$ represented by the intervals
containing each $x_i$. Furthermore, each matrix $M(x)$ is the Schur
product of a finite number of matrices $M(x,i,j)$. Each $M(x,i,j)$ is
a moment matrix of a random variable with finite range, and hence it
has finite rank. Hence every matrix $M(x)$ has finite rank, and so
$M(f,n)$ has finite rank.

\medskip

\noindent (2)$\Rightarrow$(3) is trivial.

\medskip

\noindent (3)$\Rightarrow$(1) If $f$ is identically zero, then we
regard it as the homomorphism function into the empty graph. Assume
that $f$ is not identically zero. First we prove that $f(K_1)>0$
where $K_1$ is the one-node graph. Regarding $K_1$ as an element of
the algebra $\QQ_1$, we get from the weak reflection positivity of
$f$ that $f(K_1)^2=f(K_1)\geq 0$. Now assume that $f(K_1)=0$. Let
$F\in\FF_n$ be a graph with $f(F)\neq 0$ for some natural number $n$
and let $e_n\in\FF_n$ denote graph with $n$ labeled nodes and no edge
(the unit element in $\FF_n$). By multiplicativity, $f(e_n)=0$. By
weak reflection positivity we get for every real $\lambda$ that
\[
0\leq f((F-\lambda e_n)^2)=f(F^2)-2\lambda f(F),
\]
which is a contradiction.

Replacing $f$ by $f/f(K_1)^{|V(F)|}$, we may assume that $f$ is
normalized. The matrix $E(f)$ is positive semidefinite with finite
rank. This implies that the sequence $f(K_2^n),~~n=0,1,2,\dots$ is
the moment sequence of some random variable $X$ whose values are from
a finite set. It follows that there is a number $d>0$ such that
$|f(K_2^n)|\leq d^n$ for every $n$. By Theorem \ref{chr-t2}, this
implies that there is a $[-d,d]$-graphon $W$ such that $f(F)=t(F,w)$
for all graphs $F$.

Let $(W_0,W_1,\dots)\in\MM(d)$ be the moment function sequence
representing $W$. We show that each function $W_i$ is a stepfunction.
By Theorem \ref{step}, it is enough to show that $C(t_{W_i})$ and
$B(t_{W_i})$ have finite rank. This will follow if we show that both
are submatrices of $M(2,f)$.

Let $P_{a;i}\in\GG_2$ denote the path of length $a$ in which each
edge is $i$-fold and the two endpoints are labeled by $1$ and $2$.
Let $K_{a;i}\in\GG_2$ denote the complete bipartite graph $K_{2,a}$
in which each edge is $i$-fold and the nodes from the color class
with two nodes are labeled by $1$ and $2$. It is clear from the
definitions that the $\{P_{a;i}|a\geq 2\}\times \{P_{a;i}|a\geq 2\}$
sub-matrix of $M(2,f)$ is identical with $C(t_{W_i})$ and the
$\{K_{a;i}|a\geq 0\}\times \{K_{a;i}|a\geq 0\}$ sub-matrix of
$M(2,f)$ is identical with $B(t_{W_i})$ for all $i$. This proves that
each $W_i$ is a stepfunction.

Next we argue that the $W_i$ can be considered stepfunctions with the
same steps. For every pair $x,y\in [0,1]$, $W(x,y)$ is a random
variable with values in $[-d,d]$. Let $Y$ be the random variable
which is obtained by selecting two random points $x,y$ uniformly form
$[0,1]$ and then evaluating the random variable $W(x,y)$. It is clear
that
\[
\E(Y^i)=\int_{[0,1]^2}W_i(x,y)=f(K_2^i)=\E(X^i),
\]
and thus the distribution of $Y$ is the same as the distribution of
$X$, which is concentrated on the finite set $S$. It follows that for
almost all pairs $x,y\in [0,1]$ the distribution of $W(x,y)$ is
concentrated on $S$ and at such places the first $|S|$ moments
$W_1(x,y),\dots,W_{|S|}(x,y)$ of $W(x,y)$ determine all other
moments. This means that all of the functions $W_i$ are stepfunctions
with the same steps that are intersections of the steps of
$W_1,W_2,\dots,W_{|S|}$.

Thus there is a partition $[0,1]=P_1\cup P_2\cup\dots\cup P_t$ such
that the variables $W(x,y)$ are constant on $P_i\times P_j$ for all
$1\leq i,j\leq t$. This defines the structure of a randomly weighted
graph $H$ on $\{1,2,\dots,t\}$, in which the nodeweights are the
sizes of the sets $P_i$, and the edgeweight of $ij$ is the random
variable $W(x,y)$ for any $(x,y)\in P_i\times P_j$. It is clear that
$f(F)=t(F,W)=\hom(F,H)$ for every graph $F$.
\end{proof*}

\subsection{The growth rate of connection ranks}

We have seen (Theorem \ref{THM:FLS}) that for a graph parameter
$f\in\TT_0(d)$, the connection ranks $\rk(M(f,k))$ are bounded by
$q^n$ for an appropriate $q$. What can we say about graph parameters
in the larger class $\TT_1(d)$? The following theorem gives the
answer.

\begin{theorem}\label{genfls2}
If $f$ is a weakly reflection positive and multiplicative graph
parameter, then $f$ belongs to one of the following three types.

\smallskip

{\rm(1)} $\rk(M(f,n))=\infty$ for all $n\geq 2$.

\smallskip

{\rm(2)} $\rk(M(f,n))^{1/n}\to c$ $(n\to\infty)$ with some $c\ge 1$,
and there is a weighted graph $H$ such that $f=\hom(.,H)$.

\smallskip

{\rm(3)} $\rk(M(f,n))^{1/n^2}\to c$ $(n\to\infty)$ with some $c>1$,
and there is a proper randomly weighted graph $H$ such that
$f=\hom(.,H)$.
\end{theorem}

\begin{remark}\label{REM:M1-FIN}
Finiteness of the rank of the first connection matrix $M(f,1)$ is not
enough here. In fact, let $W\in\WW$ be a function such that its
measure preserving automorphism group (the group of invertible
measure preserving maps $\varphi:~[0,1]\to[0,1]$ such that
$W(\varphi(x),\varphi(y)) =W(x,y)$) is transitive. (For example,
$W(x,y)=|x-y|$ is such a function.) Then $M(t_W,1)$ has rank $1$.
However, such functions may be far from being stepfunctions.
\end{remark}

Before proving Theorem \ref{genfls2}, we remark that the limiting
constants $c$ in (2) and (3) can be described easily, once we know
that $f$ is given by a randomly weighted graph. In the case when
$f=t(.,H)$ for an ordinary weighted graph $H$, the rank of $M(f,n)$
was described in \cite{L}. We may assume that $H$ has no twin nodes,
since we can identify twin nodes in $H$ without changing $f$. Let us
state also a description the dimension of $\PP_n/f$.

\begin{lemma}\label{LEM:OLD-RK}
Let $f=t(.,H)$, where $H$ is a weighted graph without twin nodes.
Then

\smallskip

{\rm(a)} The dimension of $\QQ_n/f$ is equal to the number of
non-equivalent maps $[n]\to V(H)$, where two maps $\varphi,\psi$ are
equivalent if there is an automorphism $\alpha$ of $H$ such that
$\varphi\alpha=\psi$.

\smallskip

{\rm(b)} The dimension of $\PP_n/f$ is equal to the number of
non-equivalent maps $[n]\to V(H)$, where two maps $\varphi,\psi$ are
equivalent if there is an isomorphism $\alpha$ between the subgraphs
of $H$ induced by ${\rm Rng}(\varphi)$ and ${\rm Rng}(\psi)$ such
that $\varphi\alpha=\psi$.
\end{lemma}

Part (a) was proved in \cite{L}. We prove part (b) in a more general
form, for randomly weighted graphs:

\begin{lemma}\label{exactrank}
Let $H$ be a randomly weighted graph with edge weights $B_{i,j}$ and
node weights $\alpha_i$, let $f=\hom(.,H)$, and let $n$ be a natural
number. Then $\dim(\PP_n/f)$ is equal to the number of different
weighted graphs $L$ on the node set $[n]$ with nodeweights $1$ for
which there is a function $\varphi:~[n]\to V(H)$ such that each edge
weight $\lambda_{i,j}$ of $L$ is an element of the range of
$B_{\varphi(i),\varphi(j)}$.
\end{lemma}

\noindent{We will not need to extend part (a) of Lemma
\ref{LEM:OLD-RK} to randomly weighted graphs, but we believe this is
possible.}

\begin{proof}
Let $f=\hom(.,H)$. Let $\Z^{(2)}_n$ denote the polynomial ring
$\R[\{z_{i,j}|1\leq i<j\leq n\}]$, which is isomorphic to the algebra
$\PP_n$. Let $A$ denote the a set of all possible pairs $(L,\varphi)$
where $L$ is a weighted graph on $\{1,2,\dots,n\}$,
$\varphi:\{1,2,\dots,n\}\mapsto V(H)$ is a function such that every
edge weight $\lambda_{i,j}$ of $L$ is in the range of
$B_{\varphi(i),\varphi(j)}$. To each element $(L,\varphi)\in A$ we
introduce the weight
\[
h(L,\varphi)=\prod_{i=1}^n\alpha_{\varphi(i)}\prod_{1\leq i<j\leq n}
P(B_{\varphi(i),\varphi(j)}=\lambda_{i,j}),
\]
which is always a positive number. By substituting the definition of
moments into the formula \ref{homtorwg} one obtains that if $p$ is an
arbitrary element of $R$ then
\[
f(p)=\sum_{(L,\varphi)\in A}h(L,\varphi)p(\lambda)
\]
where $p(\lambda)$ denotes the substitution of
$z_{i,j}=\lambda_{i,j}$ in the polynomial $p$. Note that these
substitutions are not always different for two different elements of
$A$ but after sorting the sum according to different substitutions
they can't cancel each other because the weights $h(L,\varphi)$ are
all positive. Using Lemma \ref{rankofbf} we get that $\dim(\PP_n/f)$
is equal to the number of different labeled weighted graphs $L$
occurring in the first coordinate of the elements of $A$. This is
exactly the statement of the lemma.
\end{proof}

Using this lemma, we can derive bounds on the rank of
$M(f,n)=\dim(\QQ_n/f$, where $f=\hom(.,H)$ for a randomly weighted
graph $H$. Define
\begin{equation}\label{EQ:A-DEF}
A(H) = \max\Bigl\{\frac12\sum_{u,v\in V(H)} x_u x_v\log p_{u,v}:~x\ge
0, \sum_{u\in V(H)} x_u=1\Bigr\}.
\end{equation}

\begin{lemma}\label{LEM:RANK}
Let $H$ be a randomly weighted graph, $f=\hom(.,H)$, and $n\in\N$.
Then
\[
\frac{2^{n^2A(H)}}{p(H)^{2n}} \le \dim(\PP_n/f) \le \dim(\QQ_n/f) \le
|V(H)|^n2^{n^2A(H)}.
\]
\end{lemma}

\begin{proof}
We may assume for convenience that $H$ is normalized. The upper bound
follows from an even more careful look at the proof of Theorem
\ref{genfls}, part (1)$\Rightarrow$(2). Each point $x\in[0,1]$
defines a map $\varphi:~[n]\to V(H)$, and $M(x)$ depends on this
$\varphi$ only. Each matrix $M(x,i,j)$ is a moment matrix of a random
variable with finite range of size $p_{\varphi(i),\varphi(j)}$, and
hence it has rank $p_{\varphi(i),\varphi(j)}$. Hence the rank of
$M(x)$ is at most
\[
\rk(M(x)) \le \prod_{1\le i<j\le n} p_{\varphi(i),\varphi(j)}.
\]
Let $n_u=|\varphi^{-1}(u)|$ ($u\in V(H)$), then we get
\[
\rk(M(x)) \le \prod_{u,v\in V(H)\atop u\not=v} p_{u,v}^{\frac12
n_un_v} \prod_{u\in V(H)} p_{u,u}^{\binom{n_u}{2}} \le \prod_{u,v\in
V(H)} p_{u,v}^{\frac12 n_un_v}.
\]
Here
\begin{align*}
\log \prod_{u,v\in V(H)} &p_{u,v}^{\frac12 n_un_v} =\sum_{u,v\in
V(H)} n_un_v \log p_{u,v}\\
&= n^2 \sum_{u,v\in V(H)} \frac{n_u}{n}\frac{n_v}{n} \log p_{u,v} \le
n^2 A(H),
\end{align*}
and so $\rk(M(x)) \le 2^{n^2 A(H)}$. Since there are at most
$|V(H)|^n$ different matrices $M(x)$, the upper bound follows.

To prove the lower bound, let $x\in\R^{V(H)}$ be the vector that
attains the maximum in \eqref{EQ:A-DEF}. Let $n_u$ ($u\in V(H)$) be
integers such that $|nx_u-n_u|<1$. Fix a map $\varphi:~[n]\to V(H)$
such that $|\varphi^{-1}(u)|=n_u$ for all $u$. It is clear that we
can create at least
\[
N=\prod_{u,v\in V(H)\atop u\not=v} p_{u,v}^{\frac12 n_un_v}
\prod_{u\in V(H)} p_{u,u}^{\binom{n_u}{2}}
\]
different weighted graphs on $[n]$ satisfying the condition of lemma
\ref{exactrank} by choosing the edge weights between
$\varphi^{-1}(u)$ and $\varphi^{-1}(v)$ independently from the range
of $B_{\varphi(u),\varphi(v)}$. We have
\begin{align*}
\log N &= \frac12 \sum_{u,v\in V(H)\atop u\not=v} n_un_v\log p_{u,v}
+ \sum_{u\in V(H)} \binom{n_u}{2} \log p_{u,u}\\
&= \frac12 \sum_{u,v\in V(H)} n_un_v\log p_{u,v} - \sum_{u\in V(H)}
n_u \log p_{u,u}.
\end{align*}
Here
\begin{align*}
\frac12 &\sum_{u,v\in V(H)} n_un_v\log p_{u,v} - n^2 A(H)\\
&=\frac12\sum_{u,v\in V(H)} n_un_v\log p_{u,v} - n^2
\frac12\sum_{u,v\in V(H)} x_ux_v\log p_{u,v}\\
&= \frac12\sum_{u,v\in V(H)} n_u(n_v -nx_u)\log p_{u,v} + \frac12
\sum_{u,v\in V(H)}
(n_u-nx_u)nx_v\log p_{u,v}\\
&\le \sum_{u,v\in V(H)} n_u\log p_{u,v} \le n\log p(H),
\end{align*}
and
\[
\sum_{u\in V(H)} n_u \log p_{u,u} \le n\log p(H),
\]
showing that
\[
\log N \ge n^2 A(H) - 2n\log p(H).
\]
By Lemma \ref{exactrank}, this proves that
\[
\dim(\PP_n/f)\ge N \ge \frac{2^{n^2A(H)}}{p(H)^{2n}}.
\]
\end{proof}

\begin{proof*}{of Theorem \ref{genfls2}}
Let $f$ be a weakly reflection positive and multiplicative graph
invariant. Assume that $\rk(M(f,n))<\infty$ for some integer $n\geq
2$. Since $M(f,2)$ is a submatrix of $M(f,n)$,  we have that
$\rk(M(f,2))<\infty$. By Theorem \ref{genfls} we obtain that there is
a randomly weighted graph $H$ such that $f=\hom(.,H)$. If $H$ is a
weighted graph, then by Lemma \ref{LEM:OLD-RK} it follows that
\[
\frac{|V(H)|^n}{|V(H)|!} \le \dim(\PP_n/f) \le
\dim(\QQ_n/f)  \le |V(H)|^n,
\]
and hence both $\dim(\PP_n/f)^{1/n}$ and $\dim(\QQ_n/f)^{1/n}$ tend
to $\log|V(H)|$.

On the other hand, if $H$ is a proper randomly weighted graph, then
by Lemma \ref{LEM:RANK} we have
\[
\frac{2^{A(H)}}{p(H)^{2/n}} \le \dim(\PP_n/f)^{1/n^2} \le
\dim(\QQ_n/f)^{1/n^2} \le |V(H)|^{1/n}2^{A(H)},
\]
and so both $\dim(\PP_n/f)^{1/n^2}$ and $\dim(\QQ_n/f)^{1/n^2}$ tend
to $A(H)$.
\end{proof*}

\end{document}